\documentclass[11pt,reqno]{amsart}
\usepackage{amsfonts,amsmath,amssymb}
\usepackage[all]{xy}
\newtheorem{thm}{Theorem}[section]
\newtheorem{proposition}[thm]{Proposition}
\newtheorem{corollary}[thm]{Corollary}
\newtheorem{lemma}[thm]{Lemma}
\newtheorem{definition}[thm]{Definition}
\newtheorem{example}[thm]{Example}
\newtheorem{remark}[thm]{Remark}

\usepackage{graphicx}
\usepackage{epstopdf}

\author{A.N. Sergeev}
\address{Department of Mathematics,    Saratov State University,\ \
Astrakhanskaya 83, Saratov 410012  and Moscow Institute of Physics and Technology,
9 Institutskiy per., Dolgoprudny, Moscow Region, 141701, Russin Federation}

\email{SergeevAN@info.sgu.ru,\,sergeev.an@mipt.ru}

\title [On characters of irreducible modules  for Lie superalgebra $\frak q(n)$]{ On characters of irreducible modules  for Lie superalgebra $\frak q(n)$}

\begin{document}

\maketitle

\begin{abstract} In the paper  \cite{SZ1} Su and Zhang proved a formula for  the decomposition of  characters of irreducible finite dimensional  modules over Lie superalgebra $\frak q(n)$ as the sum of Euler characters. They gave a combinatorial formula for the  coefficients in that decomposition by  generalasing  of weights and caps diagrams introduced by J. Brundan and C. Ctroppel \cite {BS}. In the present paper we give a correct   version of their combinatorial formula. Our main technical  tool is a family  of creation operators.
 \end{abstract}
 
\tableofcontents

\section{Introduction}
For the first time a formula for irreducible finite-dimensional  modules over Lie superalgebra  $\frak q(n)$ appeared in the papers I. Penkov and V. Sergenova \cite{PS,PS1}.
In the  paper \cite{B1} J. Brundan gave a formula  for the decomposition of Euler characters in terms of irreducible characters for Lie superalgebra $\frak q(n)$. 
In the paper  \cite{SZ1} Su and Zhang proved a formula for  the inverse matrix using rising operators.  They expressed characters of irreducible  modules over Lie superalgebra $\frak q(n)$ in terms of Euler characters:
\begin{equation}\label{tech}
ch\,L(\lambda)=\sum_{\mu}(-1)^{|\lambda-\mu|}2^{\frac{z(\mu)-z(\lambda)}{2}}a_{\lambda,\mu}ch\,E(\mu).
\end{equation}
Here coefficients  $a_{\lambda,\mu}$ are given by an explicit combinatorial rule. But it turned out that this formula is not true  for all irreducible modules. 
The main result of the present  paper is a correct version of  the above formula. Here is the most simple example.
Consider Lie superalgebra  $\frak{q}(5)$. Let $\lambda=(3000-3)$ and $\mu=(00000)$.  Then the corresponding cap diagram for $\lambda$ is the following one
$$
\xymatrix{{0}\ar@/^2pc/@{}[rr]_{} &1& 2& 3\ar@/^1pc/@{}[r]_{}&4}
$$
 So we see  that these caps do not include each other. Therefore according to   \cite{SZ1}   $a_{\lambda,\mu}=2$. But according to our combinatorics  (see  Example \ref{contr}) $a_{\lambda,\mu}=1$. 
 Su and Zhang used rising operators which are natural analogues of the rising operators introduced by J. Brundan \cite{B}.

Our main technical tool  is the set of creation operators. These creation operators allow to express action of the rising operators  on weight diagrams   see Lemma \ref{main2} below.

\section{Weight and cap diagrams for Lie  superalgebra $\mathfrak q(n)$}

Weight and cap diagram were introduced by Su and Zhang in  \cite{SZ1}. Our weight and  cap diagram are the same but we use a different approach to them the same as in the paper \cite{Ser2} where we enumerate the weight diagrams. The main difficulty  in the $\frak q(n)$ case  is that we have more then one cap with the initial point at $0$. So instead of numbering weights diagram we need to enumerate caps in the cap diagram.

 By multi-set (see \cite{St}) we  will mean a pair $(m,X)$ where $X$ is a set and  $m$ is a function $m : X\rightarrow \Bbb Z_{\ge 0}$. If $x\in X$ then we call $m(x)$ the multiplicity of $x$. If $X$ is finite then by $|X|$ we mean the sum $\sum_{x\in X}m(x)$.  
\begin{definition} Let $(k,A,B)$ be a triple  of  two finite  subsets in $\Bbb Z_{>0}$ and $k\in\Bbb Z_{\ge0}$. Let us define the multi-set $f$ by the following rules:
 as the set $f$ is the set of pairs $\{(x,a)\in \Bbb Z_{\ge0}\times \{\circ,\times,>,<\}\}$ such that:

$$
a=\begin{cases} \times,\,\,x\in A\cap B\\
\circ,\,\,x\notin A\cup B\cup\{0\}\\
>,\,\,x\in A\setminus B\\
<,\,\, x\in B\setminus A
\end{cases},\quad
a=\begin{cases}\circ,\,\, x=0,\,k=0\\
\times,\,\,x=0,\, k>0\end{cases}.
$$
and we suppose that multiplicity of $(0,\times)\in f$ is equal to $k$. We will denote $k=k(f)$. So we consider $f$ as a relation between $A$ and $B$. In particular
$$
f^{-1}(a)=\{x\in \Bbb Z\mid (x,a)\in f\},\quad f(x)=\{a\in \{\circ,\times,>,<\}\mid (x,a)\in f\}
$$
and we consider $f^{-1}(a)$ as a multi-set  such that multiplicity of $x$ is the same as one of $(x,a)$.
\end{definition}

We will call $f$ a weight diagram and  call $k$ as the multiplicity of zero. Let us set 
 $$
n(f) = 2|f^{-1}(\times)|+|f^{-1}(>)|+|f^{-1}(<)|.
 $$ 
 \begin{remark} We need to express $n(f)$ in terms of $n(a)$?
 \end{remark}
\begin{example}  Let $k=2$ and  $A=B=\{\,1,5\}$ then   for the  corresponding weight diagram $f$ we have the following geometric picture
$$
\xymatrix{\times^2&\times&\circ &\circ &\circ&\times&\circ&\circ}
$$
\end{example}
.
 \begin{definition}
For every  weight diagram $f$ we set 
$$
\theta(f)=S(f^{-1}(\times))-\delta(f^{-1}\{>,<\},f^{-1}(\times))+k(f)
$$
where $S(f^{-1}(\times))$ means the sum of all element in $f^{-1}(\times)$.
Let us also denote by $n_f(a)$ the number of  $x\in \Bbb Z_{\ge0}$  such that  $x<a$ and $f(a)=<,>$. Then it is easy to check that 
$$
\sum_{ f(a)=\times} n_f(a)=\delta(f^{-1}\{>,<\},f^{-1}(\times))
$$
\end{definition}

\begin{definition} Let $f$ be a weight diagram.  
 Define the cap diagram  of the first type in the following  way. Let $f(a)=\times,\,a>0$.  Take the first $c$ on the right such that $f(c)=\circ$ and  the numbers of $\times$-s and $\circ$-s in the interval $(a,c)$  are the same. Then draw the cap joining $a$ and $c$ and do that for all positive $a$.  If multiplicity of $0$ is $k$ then draw $k$ caps with the beginning  at $0$ by the following rule: take the second $c$ on the right such that the number of  $\circ$-s minus the number of $\times$-s  in the interval $(a,c)$ is equal to $1$. Then draw the cap joining $0$ and $c$;   take the second $c$ on the right such that the number of  $\circ$-s minus the number of $\times$-s  in the interval $(a,c)$ is equal to $2$ and so on. We will denote by $ \mathcal C(f)$ - the cap diagram corresponding to the weight diagram $f$. If $C\in \mathcal C(f)$ then $s(C)$ means the initial point of $C$ and $t(C)$ means the ending point of $C$.  For $C, C'\in\mathcal C(f)$ we will write  $C\subset C'$ if cap $C$ is located under cap $C'$.
 
 The cap diagram of the second type  is defined as above but the first cap with the initial point at zero is ended at the point  $c$  such that  $c$ is the first point on the right  that $f(c)=\circ$ and  the numbers of $\times$-s and $\circ$-s in the interval $(a,c)$  are the same. 
\end{definition}
\begin{example}\label{adm}  Let $f^{-1}\{>,<\}=\emptyset$,  $f^{-1}(\times)=\{0^2,1,5\}$ then we have the following first type  cap diagram 
$$
\xymatrix{{0}\ar@/^2pc/@{}[rrrr]_{}\ar@/^5pc/@{}[rrrrrrrr]_{} &1\ar@/^1pc/@{}[r]&2 &3& 4& 5\ar@/^1pc/@{}[r]&6&7&8}
$$

The  second type cap diagram  for the same weight diagram is the following one
$$
\xymatrix{{0}\ar@/^2pc/@{}[rrr]_{}\ar@/^5pc/@{}[rrrrrrr]_{} &1\ar@/^1pc/@{}[r]&2 &3& 4& 5\ar@/^1pc/@{}[r]&6&7&8}
$$
\end{example}
\vskip0.3cm

Let $\Bbb Z^n_+$ denote the set of 
$\lambda = (\lambda_1,\dots,\lambda_n) \in \Bbb Z^n$
such that $\lambda_1 \geq \dots\geq \lambda_n$ and moreover
$\lambda_r = \lambda_{r+1}$ implies $\lambda_r = 0$ for each $r = 1,\dots,n-1$.
For $\lambda \in \Bbb Z_+^n$,
let $z(\lambda)$ denote the number of $\lambda_r \:(r=1,\dots,n)$
that equal zero.  Then we define two finite subsets in $\Bbb Z_{>0}$
 $$ 
A_{\lambda}=\{\lambda_i\mid \,\lambda_i>0\},\quad  B_{\lambda}=\{|\lambda_i|\mid \,\lambda_i<0\}
 $$
  and we also have the number $k_{\lambda}=[\frac{z(\lambda)}{2}]$.  So we can define the corresponding diagram $f_{\lambda}$. The following statement is easy to check.
\begin{proposition}\label{bijec} The correspondents $\lambda\rightarrow f_{\lambda}$ is a bijection  between $\lambda\in\Bbb Z^n_+$ such that $z(\lambda)$ is odd  and the set of diagram $f$ such that $n(f)+1=n$.  The same correspondence is a bijection between $\lambda\in\Bbb Z^n_+$ such that $z(\lambda)$ is even  and the set of diagram $f$ such that $n(f)=n$
\end{proposition}

In the rest of the paper we will use the language of the weight diagrams and we will suppose  that the corresponding cap diagram  is of the first type. 

\begin{definition}  Let $f$ be a weight diagram and $r=|\mathcal C(f)|$. A numeration of the cap diagram $\mathcal C(f)$ is a bijection $\xi: [1.r]\rightarrow \mathcal C(f)$ such that conditions $\xi(i)\subset \xi(j)$ and $s(\xi(i))=s(\xi(j))=0$  imply that  $i>j$. A numeration $\xi$ is called standard if condition $i<j$ implies $s(\xi(i))\le s(\xi(j))$.  Every cap diagram has the standard numeration.We will denote such a numeration by $\xi_{st}$. We call the pair $(f,\xi)$ as a numbered weight diagram. We will denote by $s(\xi)$ the sequence $s(\xi(1)),\dots,s(\xi(r))$.
\end{definition}

\begin{remark} \label{rem} Let $\mathcal C^+(f)$ be the subset  of $\mathcal C(f)$ consisting of the caps with the initial points different from $0$. Then it is easy to see that any numeration in the above sense  uniquely defined by its restriction on $\mathcal C^+(f)$ and moreover the set  $\mathcal C^+(f)$ uniquely defined by initial points of its caps.
\end{remark}

 The rising operators  for   weight diagrams were  define by Su and Zhang (\cite{SZ1}). We defined the same type of operators  in a slightly  different form.

\begin{definition}\label{xi} Let  $f$  be  a  weight  diagram,  with a numeration $\xi$ of its cap diagram $\mathcal C(f),\,|\mathcal C(f)|=r$.  For  $i\in  [1,r]$   we  define the new weight diagram $g$  by the following rule:  we move the cross $s(\xi(i))$  to  the  point  $t(\xi(s))$.  All other crosses are still the same. If $\xi(i)=0$ then the multiplicity  of zero in the diagram $g$ is less by $1$ then  the  multiplicity of zero in $f.$  We also  get a numeration on the positive points of diagram $g$ prescribing the moving cross the same number $i$.     Therefore by Remark \ref{rem} we get the corresponding numeration on $\mathcal C(g)$. In such a case we will write $T_i(f,\xi)=(g,\eta)$.
\end{definition}

 Let $V$ be a vector space with the basis $v( g,\xi)$ where $(g,\xi)$ run over  all  numbered weight diagrams.  We also can define the action of rising operators on $V$ by the following formulae
 $$
T_i(v(g,\xi))= v((T_i(g,\xi)))
 $$
 Let also $U$ be a vector space with the basis $v(g)$ where $g$ run over all weight diagrams. We have a natural map $\sharp: V\rightarrow U$  such that $v(g,\xi)^{\sharp}=v(g)$.

 \begin{lemma}\label{prod}  Let $(g,\xi)$ be a numbered weight  diagram such that : if $C\supset C'$ then $\xi(C)<\xi(C')$. Let us    write  $s(\xi)=\{s(\xi(i_1))\le\dots\le s(\xi(i_r))\}$ where $\{i_1,\dots,i_r \}=[1,r],\,r=|\mathcal C(g)|$. Then   the following equality holds true
$$
((1+T_{r})\dots(1+T_{1})v(g,\xi))^{\sharp}=((1+T_{i_r})\dots(1+T_{i_1})v(g,\xi))^{\sharp}
$$

\end{lemma}
\begin{proof} 

It is easy to see that we can suppose that $g$ is core free. Let $C_1,\dots, C_p$ be all caps such that they are not contained in any other cap of the cap diagram $g$.
Let  $(g_1,\xi_1),\dots,(g_p,\xi_p)$ be the corresponding sub-diagrams  of $g$.   It is  also easy to see that $T_{\alpha},\,T_{\beta}$ pairwise commute  if $\alpha,\beta$ belong to different  $g_i,g_j$. Therefore it is enough  to prove Lemma for the case when in cap diagram of $g$  there is the cap which contains  any other cap. If it  is so  then $ \xi(i_1)\supset   \xi(i)$ for any $i\ne i_1$. By our assumption on $\xi$ we have $i_1<i$ for any $i\ne i_1$. Therefore $i_1=1$. Let $h$ be the diagram obtained from $g$ by deleting $\xi(1)$ and $\eta$ is the restriction of $\xi$ on to $\{2,\dots,s\}$. By induction Lemma is true for $(h,\eta)$ and therefore  it is true for $(g,\xi).$
\end{proof}

For every weight diagram $f$ we have Euler  virtual module $E(f)$ and irreducible module $L(f)$ (see \cite{B1}).

The  following Theorem   is due to Brundan (see\cite {B1}) but we formulate it here in the form due to Su and Zhang (\cite{SZ1}).

\begin{thm}  Let $|\mathcal C(g)|=r.$ Then the  following equality holds true
\begin{equation}\label{Brun}
ch\,E(g)=\sum_{f}2^{k(f)-k(g)}ch\,L(f)
\end{equation}
where  sum is taken over all $f$ such that $|\mathcal C(f)|=r,\,g=(T_r^{\varepsilon_r}\dots T_1^{\varepsilon_1}(f,\xi_{st}))^{\sharp}$, $\varepsilon_i=0,1$ for $i=1,\dots,r.$ 
\end{thm}
As a corollary of this formula we have the following equality.
\begin{corollary}\label{formula}  Let $g$ be a weight diagram, then the following equality holds true
$$
L(g)=\sum_{f}(-1)^{\theta(f)+\theta(g)}2^{k(f)-k(g)}d_{gf}E(f)
$$
 where $d_{gf}$ is the number of sequences $(i_1,\dots,i_r)\in\Bbb Z_{\ge0}^r$ such that $g=(T_1^{i_1}\dots T_r^{i_r}(f,\xi_{st}))^{\sharp}$.
\end{corollary}
\begin{proof}  Let us  set $v_f=2^{k(f)}ch\,L(f),\, e_g=2^{k(g)}E(g)$, Then   formula (\ref{Brun}) can be rewritten in the form
\begin{equation}\label{Brun2}
e_g=\sum_{f}v_f
\end{equation}
where  sum is taken over all $f$ such that $g=(T_r^{\varepsilon_r}\dots T_1^{\varepsilon_1}(f,\xi_{st}))^{\sharp}$ and $\varepsilon_i=0,1$ for $i=1,\dots,r$.
Let $v^*_f$ and $e^*_g$ are dual families then the following equality holds true
\begin{equation}\label{Brun1}
v_f^*=\sum_{g}e^*_g,\quad g=(T_r^{\varepsilon_r}\dots T_1^{\varepsilon_1}(f,\xi_{st}))^{\sharp},\, \varepsilon_i=0,1
\end{equation}
Indeed  we have  by formula (\ref{Brun2})
$$
v^*_f(e_g)=\begin{cases}1,\,\,\,g=(T_1^{\varepsilon_1}\dots T_r^{\varepsilon_r}f)^{\sharp}\\
0,\,\,\, otherwise
\end{cases}
$$
and we proved  formula (\ref{Brun1}).  Further we have  
$$
e^*_{f,\xi_{st}}=(1+T_r)\dots(1+T_1)(1+T_1)^{-1}\dots(1+T_r)^{-1}e^*_{f,\xi_{st}}
$$
$$
=\sum_{i_1,\dots,i_r}(1+T_r)\dots(1+T_1)(-1)^{i_1+\dots+i_r}T_1^{i_1}\dots T_r^{i_r}e^*_{f,\xi_{st}}
$$
$$
=\sum_{(g,\xi)=T_1^{i_1}\dots T_r^{i_r}(f,\xi_{st})}(-1)^{i_1+\dots+i_r}(1+T_r)\dots(1+T_1)e^*_{g,\xi}
$$
Now let us apply to above equality operation $\sharp$. Then we get
$$
e^*_f=(e^*_{f,\xi_{st}})^{\sharp}=\sum_{g,\xi}(-1)^{i_1+\dots+i_r}((1+T_r)\dots(1+T_1)e^*_{g,\xi})^{\sharp}
$$
$$
=\sum_{g,\xi}(-1)^{i_1+\dots+i_r}(v^*_{g,\xi})^{\sharp}=\sum_{g,\xi}(-1)^{i_1+\dots+i_r}v^*_{g}
$$
$$
=\sum_{i_1,\dots,i_r}(-1)^{i_1+\dots+i_r}v^*_{g},\,
$$
where  $g=(T_1^{i_1}\dots T_r^{i_r}(f,\xi_{st}))^{\sharp}$. Now it is not difficult to verify that  $i_1+\dots+i_r=\theta(f)+
\theta(g)$ by $\mod 2$.  Indeed it is enough to consider case $i_p=1$. Then we have only one cap. Let $a>0$ be its source and $b$ its target. Then $b-a\equiv n_f(b)-n_f(a)+1\, (\mod\, 2)$  If $a=0$ then $b-0\equiv n_f(b)+2\,(\mod \,2)$. Therefore we have 
$$
i_1+\dots+i_r\equiv\,\sum (b_i-a_i-n_f(a_i)+n_f(b_i))+k(f)-k(g)=\theta(f)+\theta(g)
$$

Therefore we have 
$$
e^{*}_f=\sum_{g=(T_1^{i_1}\dots T_r^{i_r}(f,\xi_{st}))^{\sharp}}(-1)^{\theta(f)+\theta(g)}e^{*}_g=\sum_{g}(-1)^{\theta(f)+\theta(g)}d_{g,f}v^{*}_g
$$
Therefore for transpose matrix we have
$$
v_g=\sum_{f}(-1)^{\theta(f)+\theta(g)}d_{gf}e_f
$$
where sum is taken over all $f$ such that  $g=(T_1^{i_1}\dots T_r^{i_r}(f,\xi_{st}))^{\sharp}$.
\end{proof}

\section{Creation operators}

\begin{definition} Let $f$ be a weight diagram such that $|f^{-1}(\times)|=|\mathcal C(f)|=r$ and $g$ be any weight diagram. Let us denote  by $X(g,f)$  the set of sequences $(i_1,\dots,i_r)\in\Bbb Z_{\ge0}^r$ such that $g=(T_1^{i_1}\dots T_r^{i_r}(f,\xi_{st}))^{\sharp}$.
\end{definition}
The main goal of  this paper is to give a  combinatorial formula for the number $|X(g,f)|$ and the  next definition gives the most useful tool to deal with  the rising operators. 

\begin{definition}\label{create} Let $f$ be a  weight diagram and $a\in \Bbb Z\setminus f^{-1}\{>,<\}$.  Then   $Q_a(f)$  can be defined be the following rules  (for  $a\in f^{-1}\{>,<\}$ operator  $Q_a(f)$ is not defined) :

$1)$ $ Q_a( f)^{-1}(>)=f^{-1}(>),\,\,  Q_a( f)^{-1}(<)=f^{-1}(<);\,\,$

$2)$ if $b\le a$ then $Q_a(f)(b)=f(b)$;

$3.1)$ if $a>0$,  set $M=\{b\in f^{-1}(\times)\mid b\ge a\}$. The corresponding new positions  of the elements of $M$ can be define by the following rule. Take the rightmost  element  $b$ from $M$ and   move it to  the second  $\circ $ on the right.  Then do  the same for the  rightmost  element in   $M\setminus\{b\}$ and  so on. 

$3.2)$ if $a=0$, set  $M=\{b\in f^{-1}(\times)\mid b> a\}$ The corresponding new positions  of the elements of $M$ can be define by the following rule. Take the rightmost  element  $b$ from $M$ and   move it to  the second  $\circ $ on the right.  Then do  the same for the  rightmost  element in   $M\setminus\{b\}$ and  so on. 

$k(Q_0(f))=k(f)+1.$

For any diagram $f$ we will denote by $f_{core}$  the diagram obtaining from $f$ by replacing all crosses by symbols $\circ$.

\end{definition}

\begin{remark} From geometric point of view we simply add one more cross to diagram $f$ at the position $a$.  If $a>0$ then all the  symbols  on the left on $a$ do not change their  positions. All symbols $<,>$ to the right on $a$ do not change their position also. But symbols $\times$ (including the one which was at the position $a$ if any) move to the right  on the second free symbol $
\circ$. If $a=0$ then we do not move crosses at position zero.
\end{remark}

\begin{example}
Consider the following  weight diagram $f$ where the first symbol located at point $0$
$$
\xymatrix{\times^2&\times&\circ &\circ &\circ&\times&}
$$
We want to construct it by means of creation operators.
We start with the following diagram $f_{core}$
$$
\xymatrix{\circ&\circ&\circ &\circ &\circ &\circ& }
$$
Then we  see that $Q_0(f_{core})$ is the following diagram
$$
\xymatrix{\times &\circ&\circ&\circ &\circ &\circ&} 
$$
 In the same way   for  $Q_1Q_0(f_{core})$ we have 
$$
\xymatrix{\times&\times&\circ&\circ &\circ &\circ&}
$$
Then  we see that  $Q_0Q_1Q_0(f_{core})$ is equal to
$$
\xymatrix{\times^2&\circ&\circ&\times &\circ &\circ&}
$$
Therefore $Q_1Q_0Q_1Q_0(f_{core})=f.$
It is  also easy to check  that  
$$
f=Q_1Q_3Q_0Q_0(f_{core})=Q_5Q_1Q_0Q_0(f_{core}).
$$

\end{example}
\begin{remark} Every weight diagram  can be represented in the form 
$$
f=Q_{a_r}\dots Q_{a_1}(f_{core})
$$
where $a_i$ is the initial point of the cap number $i$ with standard numeration on $\mathcal C(f)$.
\end{remark}
\begin{definition} Let us set 
$$
\omega(a,i)=\begin{cases} 0,\,\,a>0\\
1,\,\,a=0,\, i>0\\
0,\,\, a=0,\,i=0
\end{cases}
$$
 We recall that for   $a\in f^{-1}(\times)$  we   denote by $n_f(a)$ the number $x\in\Bbb Z_{\ge0}$ such that $f(x)=<,>$ and $x<a$.
\end{definition}

The following Lemma allows to express the action of the operators  $T_i$ in terms of creation operators.

\begin{lemma}\label{main2}   Let  $f=Q_{b_r}\dots Q_{b_1}(f_{core})$ where $b_i$ is the initial point of the cap number i with standard numeration
on $C(f)$. Then  the following equality holds true
\begin{equation}\label{mainq}
T_1^{i_1}\dots T_r^{i_r}(Q_{b_r}\dots Q_{b_1}(f_{core}))=Q_{a_r}\dots Q_{a_1}(f_{core})
\end{equation}
where $a_p-n_f(a_p)=b_p-n_f(b_p)+i_p+\omega(b_p,i_p),\,p=1,\dots,r.$
\end{lemma}
\begin{proof}
Let $r=1$ and $f^{-1}(\times)=\{b\}$. Then  $T_1^{i}(Q_b(f_{core}))=Q_a(f_{core})$ where
$$
a=\begin{cases} b +n_f(a)-n_f(b)+i,\,\text{if}\, b>0\\
 b +n_f(a)-n_f(b)+i+1,\, \text{if}\, b=0,i>0\\
 b +n_f(a)-n_f(b)+i,\,\text{if}\,b=0,\,i=0
\end{cases}
$$

So we see that  $a-n_f(a)=b-n_f(b)+i+\omega(b,i)$. Therefore in this case the formula is true. Let $r>1$ then we  set $res\, f= Q_{b_{r-1}}\dots Q_{b_1}(f_{core})$. Since  $f$ has the standard numeration  and according to the case $n=1$  we have $T_{r}^{i_r}(f)=Q_{a_r}(res\,f)$. By induction we have   
$$
T_1^{i_1}\dots T_{r-1}^{i_{r-1}}(res\,f)= Q_{a_{r-1}}\dots Q_{a_1}(f_{core})
$$
Therefore we need to prove the equality$$
T_1^{i_1}\dots T_{r-1}^{i_{r-1}}T_r^{i_r}(f)= Q_{a_r}T_1^{i_1}\dots T^{i_{r-1}}_{r-1}(res\,f)
$$
When we apply all other rising operators  to $T_{r}^{i_r}(f)$ then these  operators have to jump  over the cap with number $r$. But this is the same as to apply operator  $Q_{a_r}$ to the function  $T_1^{i_1}\dots T_{r-1}^{i_{r-1}}(res\,f)$. Therefore we proved equality (\ref{mainq}).
\end{proof}
\section{Character formula}
\begin{definition}\label{cond} Let $a=(a_1,\dots,a_r)$ be a sequence of  nonnegative integers.
We will denote by $Y(a)$ the set of sequences of nonnegative integers $y=(y_1,\dots,y_r)$  such that  

$1)$ $0\le y_1\le\dots\le y_r$ and if $y_i=y_{i+1}$ then $ y_i=y_{i+1}=0$

$2)$ if $a_i\ge 2$ then  $ y_i\le a_i$;

$3)$ if  $a_i=0,1$ then $y_i=a_i$
\end{definition}

\begin{corollary}  Let $f,g$  be  weight diagrams and $\xi_{st}$ is the standard numeration of $\mathcal C(f)$. Then there  exists a bijection between $X(g,f)$ 
and sequences $(a_1,\dots,a_r)$ such that  $g=Q_{a_r}\dots Q_{a_1}(f_{core})$ and   $s(\xi_{st})\in Y(a).$

\end{corollary}
\begin{proof} 
First let us prove that if $a,b\notin f^{-1}(\times)$ then $b<a$ is equivalent to $b-n(b)<a-n(a)$. Indeed if $a>b$ then $a-b>n(a)-n(b)$ and $a-n(a)>b-n(b)$. If  $a-n(a)>b-n(b)$ and $b>a$ then $b-n(b)>a-n(a)$. So we came to contradiction.

Now  let  $(i_1,\dots,i_r)$  be such that $g=(T_1^{i_1}\dots T_r^{i_r}(f,\xi_{st}))^{\sharp}$. Then by   Lemma \ref{main2}  we get the sequence   $(a_1,\dots,a_r)$ such that
 $$
 a_p-n_f(a_p)=b_p-n_f(b_p)+i_p+\omega(b_p,i_p),\,p=1,\dots,r
 $$
 Therefore $a_p-n_f(a_p)-(b_p-n_f(b_p))=i_p+\omega(b_p,i_p)\ge 0$. Therefore $a_p\ge b_p$.

So if $a_p=0$ then $b_p=0$ and $i_p=\omega(b_p,i_p)=0$. If $a_p=1$ then $b_p=0,1$. If $b_p=0$ then $i_p=\omega(b_p,i_p)=0$ and we get  a contradiction $1=0$. Therefore $b_p=1,\,i_p=\omega(b_p,i_p)=0$. If $a_p\ge2$ then as before $a_p\ge b_p$. So we proved that conditions $1),2),3)$ are necessary. Let us prove that they are sufficient. We only need to indicate the values of $(i_1,\dots,i_r)$. If $a_p=0$ then from condition $3)$ it follows that $b_p=0$ and we set $i_p=0$. If $a_p=1$ then $b_p=1$ and we set $i_p=0$. Let $a_p\ge 2$ and $b_p\le a_p$. Then  $a_p-n_f(a_p)-(b_p-n_f(b_p))=\delta\ge0$. If $\delta=0$ then we set $i_p=0$. If $\delta>0$ then we set $i_p=\delta$ for $b_p>0$ and $i_p=\delta-1$ for $b_p=0.$ Corollary is proved.
\end{proof}
\begin{example}\label{contr} Consider  diagrams $f^{-1}(\times)=\{0^2\},\,f^{-1}\{<,>\}=\emptyset$ and   $g^{-1}(<,>)=\emptyset,\,g^{-1}(\times)=\{0,3\}.$  We have $f_{core}=g_{core}$ and $f=Q_0Q_0(\emptyset).$ It is also easy to see that $g=Q_3Q_0(\emptyset)=Q_0 Q_1(\emptyset).$ Therefore we have two sequences  $(a_1,a_2)=(0,3),(1,0)$ and one sequence  $(b_1,b_2)=(0,0)$. The pair of sequences $(0,3), (0,0)$ satisfies  conditions of the above corollary but the pair  $(0,1), (0,0)$  does not.
 The cap diagram for $g$ is  the following 
$$
\xymatrix{{0}\ar@/^2pc/@{}[rr]_{} &1& 2& 3\ar@/^1pc/@{}[r]_{}&4}
$$
We see that these caps do not include each other but we have only one sequence. 
\end{example}

By Corollary of Lemma \ref{main2} in order  to find the  number of solutions of the equation  $g=(T_1^{i_1}\dots T_r^{i_n}(f,\xi_{st}))^{\sharp}$  we need to find the number of sequences $(a_1,\dots,a_r)$ satisfying the conditions of the Corollary. But the first conditions   $g=Q_{a_r}\dots Q_{a_1}(f_{core})$ is the most difficult to verify. So we replace it by another equivalent condition which is more easy to verify. First note  that every presentation $g=Q_{a_r}\dots Q_{a_1}(g_{core})$  defines the numeration on the cap diagram of $g$.

\begin{definition} Let $g=Q_{a_r}\dots Q_{a_1}(g_{core})$.  Define the numeration $\xi$ on the cap  diagram $g$ by the following rule: if $a_i>0$ then  the initial point of the cap  $\xi(i)$ is the final position of the cross which was added by  the operator $Q_{a_i}$; the caps with the initial points at $0$  get  their numeration according to Remark \ref{rem}.
\end{definition}

It turns out that we can replace  the sequence  $a_1,\dots, a_n$  by the corresponding numeration $s(\xi)=(s(\xi(1)),\dots,s(\xi(r)))$.

\begin{lemma}\label{eqv1}  Let $g=Q_{a_r}\dots Q_{a_1}(g_{core})$ and $\xi$ be the corresponding numeration of $\mathcal C(g)$. Then for diagram $f=Q_{b_r}\dots Q_{b_1}(f_{core})\,$ with the standard numeration,   relation $s(\xi_{st})\in Y(a)$ is equivalent to the relation $s(\xi_{st})\in Y(s(\xi))$. 

\end{lemma}
\begin{proof} Suppose  that $s(\xi_{st})=(b_1,\dots, b_r)\in Y(a)$. Let us prove that $s(\xi_{st})\in Y(s(\xi))$. Consider all possible cases. If $s(\xi(i))=0$ then from the definition of $\xi$ we see that $a_i\le s(\xi(i))=0$. Therefore $a_i=0$ and $b_i=0$. If  $s(\xi(i))=1$ then $a_i\le1$. If $a_i=0$ then from the definition of $\xi$ we again see that $s(\xi(i))=0\ne1$. Therefore $b_i=a_i=1$. Now let  $s(\xi(i))\ge 2$. Then  we  have  $b_i\le a_i\le  s(\xi(i))$. Therefore $s(\xi_{st})\in Y(s(\xi)).$ 
 Before going further  let us note that sequence $a=(a_1,\dots,a_r)$ has the following properties: if $a_i=0,a_j=1$ then $j>i$; there is at most one $1\le i\le r$ such that $a_i=1$. These properties easily follows from condition $s(\xi_{st})\in Y(a).$

Now let us prove the opposite statement. Let $s(\xi_{st})\in Y(s(\xi))$. If $a_i=0$ then $s(\xi(i))=0$ and $b_i=0$. If $a_i=1$ then as we suppose (we have only one $1$ and zeroes have numbers strictly less then $1$) for any $j>i$ we have $a_j\ne 0,1$. Therefore $g(i)=1$ and $b_i=1$.  Let $a_i\ge 2$. We need to  prove that $b_i\le a_i$. Suppose that there exists $i$ such that $a_i\ge 2,\,b_i\le g_i$ and $b_i>a_i$. We can suppose that $i$ is a maximal with such a property. Note that $i\ne r$ since $a_r=g_r$. Therefore  for any $j>i$ we have  $a_i<b_i<b_j\le a_j$. Therefore $a_i<a_j$ for any $j>i$  and by definition of $g$ we see that $g_i=a_i<b_i$. So we came to contradiction. Therefore $\xi_{st}\in Y(a)$ and we proved the Theorem.
\end{proof}

So we see in order to find $|X(g,f)|$ we need to describe  the numerations $\xi$ which correspond  to the relation $g=Q_{a_r}\dots Q_{a_1}(g_{core})$.
\begin{definition}\label{order}  Let us introduce the following partial order on $\mathcal C(g)$. We write  $C\dashv C'$ if one of the following conditions holds true:

$1)$  $C\supset C'$;

$2)$  $s(C)=0,\, s(C')=t(C)+1+n_g(s(C'))-n_g(t(C))$;

$3)$  $s(C)=1, s(C')=t(C)+1+n_g(s(C'))-n_g(t(C))$

A numeration $\xi$ of a cap diagram $\mathcal C(g)$  is called admissible if  it preserves the partial order. In other words  if $\xi(i)\dashv \xi(j)$ then $i<j$. 

Geometrically conditions $2)$ means that there is no symbol $\circ$ between  the final point of the cap $C$ and the initial point of the cap $C'$ and the same for condition $3)$.
\end{definition}

\begin{lemma}
  Conditions $g=Q_{a_r}\dots Q_{a_1}(f_{core})$ and $s(\xi_{st})\in Y(a)$ are equivalent to the conditions $\xi$ is admissible and $s(\xi_{st})\in Y(s(\xi)).$ 
\end{lemma}

\begin{proof} Suppose that $g=Q_{a_r}\dots Q_{a_1}(f_{core})$ and $\xi$ is the corresponding numeration. The cross with number $i$ appears as  the result of $Q_{a_i}\dots Q_{a_1}(f_{core})$  and the corresponding cap  does not contained any previous caps. Therefore relation $j<i$ implies  that $\xi(i)$ does not contain $\xi(j)$. Therefore if $C\supset C'$ then $\xi^{-1}(C)<\xi^{-1}(C')$. 

As we have already seen the sequence   $a_1,\dots, a_r$ 
has the following properties: if $a_i=0,a_j=1$ then $j>i$; there is at most one $1\le i\le r$ such that $a_i=1$.

Suppose now  that   for some $C,C'\in\mathcal C(g)$ we have   $s(C)=0,\, s(C')=t(C)+1+n_g(s(C'))-n_g(t(C))$ and $i=\xi^{-1}(C)>\xi^{-1}(C')=j$. First we see that $a_j=0$ and  we know that  $a_j=h(j)$ where $h=Q_{a_j}\circ \dots\circ Q_{a_1}(\emptyset)$. Therefore  in order to define $a_j$ we need to delete all caps with numbers grater then $j$ Let us delete first  all caps with the numbers grater or equal $i$. Since $j<i$  the cap number $i$ should be deleted. After that the  initial point of cap number $j$ becomes $1$ and it still be the same while we deleting the caps with numbers $i-1,\dots, j+1$. Therefore $a_j=1$, $a_i=0$ and $i>j$. This  contradicts properties of $a$. 
In the same way we can prove the third property from Definition \ref{order}. Therefore $\xi$ is admissible and by  Lemma \ref{eqv1} $s(\xi_{st})\in Y(s(\xi)).$

Now let $\xi$ be an admissible numeration on $\mathcal C(g)$ and $s(\xi_{st})\in Y(s(\xi)).$ 
 Let us prove by induction that there exist the sequence $a_1,\dots,a_r$ 
such that  $g=Q_{a_r}\dots Q_{a_1}(g_{core})$   and the corresponding admissible numeration is $\xi.$ By definition of $Q_{a_r}$ we have $a_r=\xi(r)$.  Let $C\in \mathcal C(g)$ is the cap  such that $s(C)=a_r$. Since $\xi$ is admissible we see that if  integer $c\ne a_r=s(C),\,t(C)$   and if  it is located   under the cap $C$ then $c\in g^{-1}(>,<)$.  

Define the  weight diagram  $h$ by the following rule: 
delete  one cross from  diagram $g$ at the position $a_r$;  if $a_r>0$ then all the  symbols  on the left on $a_r$ do not change their  positions. All symbols $<,>$ to the right on $a_r$ do not change their position also. All  symbols $\times$ on the right on $a_r$  move to the left   on the second free symbol $\circ$.

It is easy to see   that $Q_{a_r}(h)=g.$ Besides under inductive assumption we have $h=Q_{a_{r-1}}\dots Q_{a_1}(g_{core})$. Therefore 
$g=Q_{a_r}\dots Q_{a_1}(g_{core})$ and again by induction for given $\xi$ there is only one sequence $a=(a_1,\dots,a_r).$
By Lemma  \ref{eqv1} we have $\xi_{st}\in Y(a).$ Lemma is proved.
\end{proof} 
The following corollary easily follows from above Lemma.
\begin{corollary} Let $f,g$ be  weight diagrams such that $g_{core}=f_{core}$. Then $|X(g,f)|$ is equal to the number of  admissible numerations  $\xi$    of $\mathcal C(g)$  such that  $\xi_{st}\in Y(\xi)$  where $\xi_{st}$ is the standard numeration of $f$. 

We will denote the set of such  pairs by $Z(g,f)$  for fixed $f$ and by $Z(g)$ we will denote  the  set of all   such pairs for fixed $g$ and all possible $f$.  So  the following equality holds true:
$$
L(g)=\sum_{(\xi,f)\in Z(g)}(-1)^{\theta(g)+\theta(f)}2^{k(f)-k(g)}E(f)
$$

\end{corollary}

\begin{example}\label{adm0} 
 Let  $g$ be the same weight diagram as in  Example \ref{adm} and  $\mathcal C(g)=\{C_1,C_2,C_3,C_4,C_5\}$ all its caps with the standard numeration. We have  has  the following partial order
 $$
 C_1\dashv C_2\dashv C_3,\, C_2\dashv C_4. 
 $$

Therefore we get two  admissible numerations $\xi_1$ and $\xi_2$.

\begin{equation}\label{adm1}
\xymatrix{{0}\ar@/^2pc/@{}[rrrr]_{2}\ar@/^5pc/@{}[rrrrrrrr]_{1} &1\ar@/^1pc/@{}[r]_3&2 &3& 4& 5\ar@/^1pc/@{}[r]_4&6&7&8}
\end{equation}
\vskip0.5cm
\begin{equation}\label{adm2}
\xymatrix{{0}\ar@/^2pc/@{}[rrrr]_{2}\ar@/^5pc/@{}[rrrrrrrr]_{1} &1\ar@/^1pc/@{}[r]_4&2 &3& 4& 5\ar@/^1pc/@{}[r]_3&6&7&8}
\end{equation}

So we see that  $s(\xi_1)=(0,0,1,5)$,\, $s(\xi_2)=(0,0,5,1)$.  Let us describe the pairs such that $(\xi,f)\in Z(g)$
Let   $f=(x_1,x_2,x_3,x_4)$ and  $0\le x_1\le x_2\le x_3\le x_4$ then  for $s(\xi_1)$ we have 
$$
x_1=x_2=0,\,\, x_3=1,\, 2\le x_4\le5.
$$
For $s(\xi_2)$ we  have  only one pair
$$
x_1=x_2=0,\,\, x_3=0,\, x_4=1. 
$$
Therefore we have 
$$
ch\,L(0^2,1,5)=E(0^2,1,5)-E(0^2,2,5)+E(0^2,3,5)-E(0^2,4,5)+2E(0^3,1)
$$

\end{example}

Now let us give a description of $d_{g,f}$ in terms of  ordered sets. If $X$ is a usual set and $(Y,m)$   is a multi-set  then $\varphi : X\rightarrow Y$ we  call a bijection if it is a surjection and $|\varphi^{-1}(y)|=m(y)$ for every $y\in Y$.

\begin{definition}\label{phi} Let $g,f$ be two weight diagrams such that $f_{core}=g_{core}$. Let us denote by  $H(g,f)$  the set  of bijections $\varphi:\mathcal C(g)\rightarrow f^{-1}(\times)$  such that: 

$1)$  for $C\in \mathcal C(g)$ we have $\varphi(C)\le s(C)$;

$2)$ if $C\dashv C'$ then $\varphi(C)\le\varphi(C')$;

$3)$ if $s(C)=1$  then $\varphi (C)=1$.

For an ordered set $S,\,|S|=p$ we will denote by $K(S)$ the set of monotonic bijections from $S$ to integer segment  $[1, p]$ with the natural order. 

 Let us also define for any $\varphi\in H(g,f)$ the number $K(\varphi)$ by the following rule

$1)$ $m(\varphi)=1$ if $\varphi^{-1}(0)=\emptyset$;

$2)$ $m(\varphi)=|K(\varphi^{-1}(0)|$ if $\varphi^{-1}(0)\ne\emptyset.$
\end{definition}
\begin{proposition} \label{ordered} The following equality holds true
$$
d_{g,f}=\sum_{\varphi\in H(g,f)} m(\varphi)
$$
\end{proposition}
\begin{proof}   Consider  the map 
$$
\alpha :\mathcal C(g)\rightarrow\mathcal C(f),\quad \beta(\xi)=\xi_{st}\xi^{-1}=\psi
$$
It is easy to check that $\alpha$ is a bijection between $Z(g,f)$ and the set $W(g,f)$ of $\psi:\mathcal C(g)\rightarrow\mathcal C(f)$ such that 

$1)$ $s\psi(C)\le s(C)$;

$2)$ if $ C\dashv C'$ then $s\psi(C)\le s\psi(C')$;

$3)$ if $s(C)=1$ then $s\psi(C)=1$;

$4)$ if $C\dashv C'$ and  $s(\psi(C)=s\psi(C')=0$ then $\psi(C)\dashv\psi(C')$.
Further  we have a map $S: W(g,f)\rightarrow H(g,f)$ such that $S(\psi)=s\circ\psi$. Clearly this map is a surjection. So we have 
$$
|Z(g,f)|=|W(g,f)|=\sum_{\varphi\in H(g,f)}|S^{-1}(\varphi)|= |H_0(g,f)|+\sum_{\varphi\in H_1(g,f)}|K(\varphi^{-1}(0))|
$$
$$
=\sum_{\varphi\in H(g,f)}m(\varphi)
$$
where $H_0(g,f)$ is the set of $\varphi\in H(g,f)$ such that $|\varphi^{-1}(0)|=0$ and $H_{1}(g,f)$ is its complement in $H(g,f)$.
\end{proof}

\begin{definition} Let us set  $H(g)=\cup_{f}H(g,f)$. Then it is easy to check that  the set of $\varphi: \mathcal C(g)\rightarrow \Bbb Z_{\ge0}$ such that $\varphi$ satisfied conditions $1),2),3)$ from Definition \ref{phi} and  if $\varphi(C)=\varphi(C')>0$ then $C=C'$.
Let us also denote by  $\varphi(g)$ the diagram such that $\varphi(g)_{core}=g_{core}$ and $\varphi(g)^{-1}(\times)$ is equal to the image of $\varphi$ as the set  and the multiplicity of zero is $|\varphi^{-1}(0)|$.

\end{definition}

\begin{corollary}
 The following equality holds true
$$
ch\, L(g)=\sum_{\varphi\in H(g)}(-1)^{\theta(\varphi)}2^{k(\varphi)}m(\varphi) ch\,E(\varphi(g))
$$ 
  where  $\theta(\varphi)=\theta(g)+\theta(\varphi(g))$ and $k(\varphi)=k(\varphi(g))-k(g)$.
 \end{corollary} 
A proof  easily follows from Proposition \ref{ordered}.

In the case of cap diagram of the second type we have the  theorem   below which can be proved  by the same way as for diagram of the first type.  Let $\mathcal C(g)$ be the cap diagram of the second type for a weight diagram $g$. We will call the cap   $C\in \mathcal C(g)$  special if $s(C)=0$ and for any $C'$ such that $s(C')=0$ we have $C\subset C'$.

\begin{definition}\label{order1}  Let us introduce the following partial order on $\mathcal C(g)$. We write  $C\dashv C'$ if one of the following conditions holds true:

$1)$  $C\supset C'$;

$2)$  $s(C)=0$, $C$ is not special and $ s(C')=t(C)+1+n_g(s(C'))-n_g(t(C))$;

A numeration $\xi$ of a cap diagram $\mathcal C(g)$  is called admissible if  it preserves the partial order. In other words  if $\xi(i)\dashv \xi(j)$ then $i<j$. 
\end{definition}

\begin{definition}\label{phi1} Let $g$ be a  weight diagrams and    $H(g)$ be the set  of $\varphi:\mathcal C(g)\rightarrow \Bbb Z_{\ge 0}$  such that: 

$1)$  for $C\in \mathcal C(g)$ we have $\varphi(C)\le s(C)$;

$2)$ if $C\dashv C'$ then $\varphi(C)\le\varphi(C')$;

$3)$  if $\varphi(C)=\varphi(C')>0$ then $C=C'$.
Let us also denote by  $\varphi(g)$ the diagram such that $\varphi(g)_{core}=g_{core}$ and $\varphi(g)^{-1}(\times)$ is equal to the image of $\varphi$ as the set  and the multiplicity of zero is $|\varphi^{-1}(0)|$.
\end{definition}

\begin{thm} 
 The following equality holds true
$$
ch\, L(g)=\sum_{\varphi\in H(g)}(-1)^{\theta(\varphi)}2^{k(\varphi)}m(\varphi)ch\,E(\varphi(g))
$$ 
  where  $\theta(\varphi)=\theta(g)+\theta(\varphi(g))$ and $k(\varphi)=k(\varphi(g))-k(g)$. 
\end{thm}

\begin{example} 
Let  us consider  the atypical block of Lie of superalgebra $\frak q(2)$. The corresponding diagrams are : $|f^{-1}(\times)|=1,\,|f^{-1}(>)|=|f^{-1}(<)|=0$.  For all such diagram we have $k(f)=0,2$ and we number them   simply by $a=f^{-1}(\times)$.  Therefore we have second type of cap diagram and the following character formula
\begin{gather*}
L(a)=\sum_{0\le b\le a}(-1)^{a+b}E(b)+(-1)^{a+1}2E(0)
\end{gather*}
where $a\ge0$.
\end{example}
\begin{example}

Now let us  consider the case $\frak q(3)$. For the atypical block  we have the same set of diagram as above but the corresponding cap diagram is of the first type. Therefore we have 
\begin{gather*}
L(0)=E(0),\,L(1)=E(1)\\
L(a)=\sum_{1\le b\le a}(-1)^{a+b}E(b)+(-1)^{a+1}2E(0)
\end{gather*}
where $a\ge2$
\end{example}
{\bf Acknowledgement.}
This work is done at the Center of Pure Mathematics MIPT. It is financially supported by
Russian Science Foundation grant 26-11-00115.

{\bf Conflict of Interest}
 The author has no conflict of interest to declare that are relevant to this article.


\begin{thebibliography}{99}

\bibitem{B}
Brundan, J.: Kazhdan-Lusztig polynomials and character formulae for Lie superalgebra $\frak{gl}(m|n).$
J. Amer. Math. Soc. 16(1), (2003)

\bibitem{B1}
J. Brundan, Kazhdan-Lusztig polynomials and character formulae for the Lie superalgebra q(n),
Adv. Math. 182 (2004) 28-77.

\bibitem{BS} 
Brundan, J., Stroppel, C.: Highest weight categories arising from Khovanov's diagram algebra IV:
the general linear supergroup. J. Eur. Math. Soc. 14, (2012)
 \bibitem{PS}
I. Penkov and V. Serganova, Characters of irreducible G-modules and cohomology of G/P for
the supergroup G= Q(N), J. Math. Sci., 84 (1997), 1382--1412.
\bibitem{PS1}
I. Penkov and V. Serganova, 
Characters of Finite-Dimensional Irreducible  $q(n)$-Modules, Letters in Mathematical Physics 40: 147--158, 1997. 

\bibitem{Ser2}
A. Sergeev, Combinatorics of Irreducible Characters for Lie Superalgebra $\frak{gl}(m,n).$
 Transformation Groups 31, 2021--2043 (2026). https://doi.org/10.1007/s00031-026-09961-3
\bibitem{SZ1} 
Y. Su and R. B. Zhang, Character and dimension formulae for queer Lie superalgebra, Comm. Math. Phys. 333
(2015), no. 3, 1465--1481.
\bibitem{St}
R.P.  Stanley Enumerative combinatorics.Vol. 1, 2nd ed. Cambridge University Press. 2012.
\end{thebibliography}
\end{document}